\documentclass[12pt]{amsart}
\numberwithin{equation}{section}
\usepackage{amsmath,amssymb,amsthm,mathtools,bm}
\usepackage{mathrsfs}
\usepackage{geometry}
\usepackage{xcolor}
\usepackage{hyperref}
\usepackage{booktabs,longtable,array}
\usepackage{enumitem}
\hypersetup{colorlinks=true,linkcolor=blue,citecolor=blue,urlcolor=blue}
\setlist{nosep}

\theoremstyle{definition}
\newtheorem{theorem}{Theorem}[section]
\newtheorem{proposition}[theorem]{Proposition}
\newtheorem{lemma}[theorem]{Lemma}
\newtheorem{corollary}[theorem]{Corollary}
\theoremstyle{definition}
\newtheorem{definition}[theorem]{Definition}

\newtheorem{example}[theorem]{Example}

\newcommand{\kk}{\mathbb K}
\newcommand{\Umir}{\mathbf U^{\mathrm{mir}}_v(2)}
\newcommand{\Ufd}{\mathbf U(2,d)}
\newcommand{\MS}{\mathcal{MS}_v(2,d)}
\newcommand{\eps}{\varepsilon}
\newcommand{\te}{\widetilde e}
\newcommand{\tf}{\widetilde f}
\newcommand{\tk}{\widetilde k}
\newcommand{\tell}{\widetilde\ell}

\newcommand{\End}{\operatorname{End}}
\newcommand{\Ker}{\operatorname{Ker}}
\newcommand{\Mat}{\operatorname{Mat}}

\newcommand{\ro}{\operatorname{ro}}
\newcommand{\co}{\operatorname{co}}
\newcommand{\cF}{\mathcal F}
\newcommand{\cO}{\mathcal O}
\newcommand{\Iset}{\mathcal I}
\newcommand{\Ifd}{\mathbf{I}(2,d)}
\newcommand{\be}{\boldsymbol e}
\newcommand{\bfm}{\boldsymbol f}
\newcommand{\bell}{\boldsymbol\ell}
\newcommand{\beps}{\boldsymbol\varepsilon}

\title{Presenting mirabolic quantum Schur algebras $\mathcal{MS}_v(2,d)$}
\author{Hongjia Chen}
\address{School of Mathematical Sciences,
	University of Science and Technology of China, Hefei 230026, China}
\email{hjchen@ustc.edu.cn (H. Chen)}

\author{Jian Chen}
\address{School of Mathematical Sciences,
	University of Science and Technology of China, Hefei 230026, China}
\email{JianChen96@outlook.com (J. Chen)}
\date{}

\begin{document}

\begin{abstract}
Let $v$ be an indeterminate. 
We give a presentation of the mirabolic quantum Schur algebra $\mathcal{MS}_v(2,d)$ in terms of generators and defining relations. 
More precisely, we determine the kernel of the natural epimorphism from Rosso's mirabolic quantum $\mathfrak{sl}_2$, 
$\mathbf U^{\mathrm{mir}}_v(2)=\langle e,f,k^{\pm 1},\ell\rangle$, onto $\mathcal{MS}_v(2,d)$. 
Setting $P_d(k)=\prod_{r=0}^{d}(k-v^{2r-d})$ and $Q_d(k)=\prod_{r=1}^{d}(k-v^{2r-d})$, we prove
$\mathcal{MS}_v(2,d)\cong \mathbf U^{\mathrm{mir}}_v(2)/ \left\langle P_d(k),(\ell-1)Q_d(k)\right\rangle$. 
We further determine the split Wedderburn decomposition of this quotient and provide an equivalent presentation in terms of weight idempotents.
\end{abstract}

\maketitle
\tableofcontents

\section{Introduction}
\subsection{Background}
The quantum Schur algebras introduced by Dipper and James provide a link between the Iwahori-Hecke algebras and quantum groups via the quantum Schur-Weyl duality \cite{Jim86,DJ89}. Beilinson, Lusztig and MacPherson realised these algebras as convolution algebras associated with pairs of partial flags. Their stabilization procedure also constructs the modified quantum group \cite{BLM90}. This construction gives explicit bases and multiplication formulas, together with natural epimorphisms from quantum groups onto finite quantum Schur algebras. 
The problem of presenting quantum Schur algebras was initiated by Doty-Giaquinto in \cite{DG02} using these epimorphisms 
(see \cite{DP03} for a different approach). 
Further Schur-type algebras are discussed in \cite{Dot03,DG07,Wad11,DW15,FS26,CD26}. These results show that a finite quantum Schur algebra may be viewed as a quantum group with additional relations.

The same point of view has been useful beyond the type $A$. BLM-type constructions have been extended to affine quantum groups and to quantum symmetric pairs (for example, see \cite{DF15,BKLW18,LL21,DW22,FLLLW20,FLLLW23,LY26}). In these settings, a central problem is to determine the kernel of a natural homomorphism from a quantum group or an $\imath$-quantum group to a finite quantum Schur type algebra. 
In the standard type $A$ presentations, the additional defining relations may be chosen in the Cartan part of the quantum group. However, the situation can be different when the source algebra has an extra generator. 
In the $\imath$-quantum group setting, the usual diagonal relations must be supplemented by relations whose images involve tridiagonal matrices \cite{CD26}. Rank-one polynomial relations for an additional generator also occur in \cite{Li20}. 
These examples show that an additional generator may force defining relations that cannot be recovered from Cartan part alone.

Rosso introduced a mirabolic analogue of the BLM construction by replacing a pair of partial flags with a triple consisting of two partial flags and a vector \cite{Ros18}. The orbits are indexed by decorated matrices, following the orbit theory in \cite{MWZ99,Tra09}. 
For two-step flags, explicit multiplication formulas and five Chevalley-type generators are obtained in \cite{Ros18}. The same paper defines the mirabolic quantum $\mathfrak{sl}_2$, denoted here by $\Umir$, and proves a PBW theorem. A classification of all finite-dimensional simple modules and the semisimplicity of the finite-dimensional module category are also established there. The corresponding Schur-Weyl problem is formulated using the mirabolic Hecke algebra of \cite{Ros14}. 
The mirabolic BLM construction and its representation theory for general rank have been further studied in \cite{FZM25,GR26}.

For each $d$, the geometric generators give a natural epimorphism from $\Umir$ onto the finite convolution algebra $\MS$. 
The universal relations in $\Umir$ do not determine this finite quotient. Rosso points out that the finite quotient satisfies further relations depending on $d$, such as  $e^{d+1}=f^{d+1}=0$, but does not give the full defining ideal \cite[\S 4]{Ros18}. Thus, the presentation problem for mirabolic quantum Schur algebras is natural but nontrivial. 
The purpose of this paper is to determine this ideal explicitly.  
The answer consists of a Cartan polynomial relation together with an extra relation involving the idempotent generator $\ell$. 
Thus, even for two-step flags, the mirabolic finite quotient is not obtained from $\Umir$ by a Cartan truncation alone.

\subsection{Notations and main results}
Throughout the paper, $v$ is an indeterminate, $\kk=\mathbb C(v)$, and $d\in\mathbb N=\{0,1,2,\ldots\}$. We write $\MS$ for the generic mirabolic quantum Schur algebra attached to two-step flags in dimension $d$. Its geometric Chevalley-type generators are denoted by $\te,\tf,\tk,\tk^{-1},\tell$. The generators of $\Umir$ are denoted by $e,f,k,k^{-1},\ell$.

The main result is the following presentation theorem.
\begin{theorem}\label{thm:intro}
Let $P_d(k)=\prod_{r=0}^{d}(k-v^{2r-d})$, $Q_d(k)=\prod_{r=1}^{d}(k-v^{2r-d})$. The natural epimorphism
$$
\pi_d:\Umir\longrightarrow\MS,
\qquad
e\mapsto\te,\quad f\mapsto\tf,\quad k^{\pm 1}\mapsto{\tk}^{\pm 1},\quad \ell\mapsto\tell,
$$
has kernel
$$
\Ker\pi_d=
\left\langle P_d(k),(\ell-1)Q_d(k)\right\rangle.
$$
Equivalently, there is a unique $\kk$-algebra isomorphism
$$
\mathbf U(2,d):=
\Umir\Big/\left\langle P_d(k),(\ell-1)Q_d(k)\right\rangle
\xrightarrow{\ \sim\ }
\MS
$$
which sends $e,f,k^{\pm 1},\ell$ to $\te,\tf,{\tk}^{\pm 1},\tell$, respectively.
\end{theorem}

The two new relations play different roles.  The polynomial $P_d$ gives the allowed eigenvalues $v^{-d},v^{-d+2},\ldots,v^{d-2},v^d$ of $k$ on a $\Ufd$-module. So we call it the Cartan relation in this paper. The second relation is equivalent to $(\ell-1)\varepsilon_0=0$, where $\varepsilon_0=\frac{Q_d(k)}{Q_d(v^{-d})}$ acts as the projection onto the $v^{-d}$-eigenspace of $k$. 
Thus, the second relation requires $\ell$ to act as the identity on the $v^{-d}$-eigenspace.

We prove the theorem in four steps. First, the decorated-orbit basis gives $\dim_{\kk}\MS=(d+1)^3$. 
Second, we use the relation $P_d(k)=0$ to construct Lagrange idempotents $\varepsilon_0,\ldots,\varepsilon_d$ in $\Ufd$. Their shift relations with $e$ and $f$ imply $e^{d+1}=f^{d+1}=0$. Rosso's PBW basis then shows that $\Ufd$ is finite-dimensional. Third, Rosso's classification allows us to determine exactly which simple modules satisfy both new relations. This gives the split Wedderburn decomposition of $\Ufd$ and the dimension formula $\dim_{\kk}\Ufd=(d+1)^3$. 
Finally, the natural map $\Ufd\twoheadrightarrow\MS$ is an isomorphism by dimension.

The parameter $v$ remains generic throughout. This assumption ensures that the elements $v^{2r-d}$, $0\le r\le d$, are distinct.

\subsection{Organization} 
The paper is organised as follows. Section~\ref{sec:geom} recalls the decorated-orbit basis and proves the two new relations in the mirabolic quantum Schur algebras. Section~\ref{sec:quot} studies the quotient $\Ufd$ and its Lagrange idempotents. Section~\ref{sec:simple} determines all simple $\Ufd$-modules. Section~\ref{sec:main} establishes split semisimplicity, computes the dimension, and proves the main theorem. Section~\ref{sec:idemp} gives an equivalent presentation by weight idempotents.

\subsection*{Acknowledgment}
This paper is partially supported by the National Key R\&D Program of China (2024YFA1013802), the NSF of China (12671043) and the Innovation Program for Quantum Science and Technology (2021ZD0302902).

\section{Mirabolic quantum Schur algebras \texorpdfstring{$\MS$}{MS}}\label{sec:geom}
\subsection{Decorated orbits and the convolution basis}\label{subsec:orbits}

Let $q$ be a prime power and let $V$ be a $d$-dimensional vector space over $\mathbb F_q$. Set
$$
\cF(2,d)=
\{F=(0=F_0\subseteq F_1\subseteq F_2=V)\}.
$$
The group $G_d=\mathrm{GL}(V)$ acts diagonally on
$$
\mathfrak X_d=\cF(2,d)\times\cF(2,d)\times V.
$$
Following \cite[\S~2]{Ros18}, the space of $G_d$-invariant functions on $\mathfrak X_d$ is equipped with a convolution product. The structure constants with respect to the orbit basis lie in $\mathbb Z[q]$. Substituting $q=v^2$ and extending scalars to $\kk$ gives the generic mirabolic quantum Schur algebra $\MS$.

Let
$$
\Theta_{2,d}=
\left\{
A=(a_{ij})\in\operatorname{Mat}_2(\mathbb N)
\ \middle|\
\sum_{1\le i,j\le2}a_{ij}=d
\right\}.
$$
For $A\in\Theta_{2,d}$, write
$$
\ro(A)=(a_{11}+a_{12},a_{21}+a_{22}),
\qquad
\co(A)=(a_{11}+a_{21},a_{12}+a_{22}).
$$
For two-step flags, a decoration is one of the sets
$$
\varnothing,\quad
\{(1,1)\},\quad
\{(1,2)\},\quad
\{(2,1)\},\quad
\{(1,2),(2,1)\},\quad
\{(2,2)\}.
$$
A decorated matrix is a pair $(A,\Delta)$ such that $A\in\Theta_{2,d}$ and $a_{ij}>0$ for every $(i,j)\in\Delta$. Let $\Xi_{2,d}$ be the set of all such pairs.

\begin{example}\label{ex:sixdecor}
Let $\Delta_1=\varnothing$, $
\Delta_2=\{(1,1)\}$, $
\Delta_3=\{(1,2)\}$, $
\Delta_4=\{(2,1)\}$, $
\Delta_5=\{(1,2),(2,1)\}$, $
\Delta_6=\{(2,2)\}$, and
$$
A=\begin{pmatrix}a_{11}&a_{12}\\a_{21}&a_{22}\end{pmatrix}\in\Theta_{2,d}.
$$
The six possible decorated matrices can be displayed as follows:
$$
\begin{array}{ccc}
   (A,\Delta_1)=\begin{pmatrix}a_{11}&a_{12}\\a_{21}&a_{22}\end{pmatrix},  &(A,\Delta_2)=\begin{pmatrix}\boxed{a_{11}}&a_{12}\\a_{21}&a_{22}\end{pmatrix},  &(A,\Delta_3)=\begin{pmatrix}a_{11}&\boxed{a_{12}}\\a_{21}&a_{22}\end{pmatrix},\\
    (A,\Delta_4)=\begin{pmatrix}a_{11}&a_{12}\\\boxed{a_{21}}&a_{22}\end{pmatrix}, &(A,\Delta_5)=\begin{pmatrix}a_{11}&\boxed{a_{12}}\\\boxed{a_{21}}&a_{22}\end{pmatrix},  &(A,\Delta_6)=\begin{pmatrix}a_{11}&a_{12}\\a_{21}&\boxed{a_{22}}\end{pmatrix}.
\end{array}
$$
A boxed entry records a position belonging to $\Delta$, and every boxed entry is required to be positive. These are precisely the six cases in \cite[(3.1)]{Ros18}.
\end{example}

The decorated matrices $(A,\Delta)\in\Xi_{2,d}$ parametrize the diagonal $G_d$-orbits on $\cF(2,d)\times\cF(2,d)\times V$. We denote the orbit corresponding to $(A,\Delta)$ by $\cO_{A,\Delta}$ and its characteristic function by $T_{A,\Delta}$. The elements
$$
\{T_{A,\Delta}\mid(A,\Delta)\in\Xi_{2,d}\}
$$
form a $\kk$-basis of $\MS$ (cf. \cite[Definition~2.2 \& Section~3]{Ros18} or \cite{MWZ99}). The convolution product satisfies the following condition:
$$
T_{A,\Delta}T_{B,\Gamma}=0
\qquad\text{if}\qquad
\co(A)\ne\ro(B).
$$
For a diagonal matrix $D\in\Theta_{2,d}$, one also has
\begin{equation}\label{eq:conviden}
    T_{D,\varnothing}T_{A,\Delta}
=\delta_{\co(D),\ro(A)}T_{A,\Delta},\qquad T_{A,\Delta}T_{D,\varnothing}
=\delta_{\co(A),\ro(D)}T_{A,\Delta}.
\end{equation}
These identities follow directly from the definition of convolution (cf. \cite[(2.3)]{Ros18}).

\begin{lemma}\label{lem:geomdim}
For every $d\in\mathbb N$,
$$
\dim_{\kk}\MS=(d+1)^3.
$$
\end{lemma}
\begin{proof}
    If $\Delta=\varnothing$, every matrix in $\Theta_{2,d}$ is allowed, so this case contributes $\binom{d+3}{3}$ basis elements. For each of the four choices of $\Delta$ with $|\Delta|=1$, one prescribed entry of $A$ is positive. After subtracting $1$ from that entry, there remain $\binom{d+2}{3}$ choices. Finally, for $\Delta=\{(1,2),(2,1)\}$, both $a_{12}$ and $a_{21}$ are positive, and this case contributes $\binom{d+1}{3}$ choices. Therefore
$$
\dim_{\kk}\MS
=\binom{d+3}{3}+4\binom{d+2}{3}+\binom{d+1}{3}=(d+1)^3.  
$$
\end{proof}

\subsection{Weight idempotents and geometric generators}\label{subsec:geomgen}

Let $E_{ij}$ be the $2\times2$ matrix unit with 1 in the $(i,j)$-entry and zeros everywhere else. For $0\le r\le d$, put
$$
D(r,d-r)=
\begin{pmatrix}
r&0\\
0&d-r
\end{pmatrix},
\qquad
\mathbf1_r=T_{D(r,d-r),\varnothing}.
$$
For $0\le r\le d-1$, put
$$
e_r=T_{D(r,d-r-1)+E_{12},\varnothing},
\qquad
f_r=T_{D(r,d-r-1)+E_{21},\varnothing},
$$
and, for $1\le r\le d$, put
$$
x_r=T_{D(r,d-r),\{(1,1)\}}.
$$
The convolution identities \eqref{eq:conviden} above give
\begin{equation}\label{eq:orthbfirs}
    \mathbf1_r\mathbf1_s=\delta_{r,s}\mathbf1_r,
\qquad
\sum_{r=0}^{d}\mathbf1_r=1.
\end{equation}

The geometric Chevalley-type generators introduced in \cite[Theorem~3.5]{Ros18} are
$$
\te=\sum_{r=0}^{d-1}v^{-r}e_r,
\qquad
\tf=\sum_{r=0}^{d-1}v^{1+r-d}f_r,
$$
$$
\tk=\sum_{r=0}^{d}v^{2r-d}\mathbf1_r,
\qquad
\tk^{-1}=\sum_{r=0}^{d}v^{d-2r}\mathbf1_r,
$$
and
$$
\tell=\mathbf1_0+
\sum_{r=1}^{d}v^{-2r}(\mathbf1_r+x_r).
$$
\begin{theorem}\label{thm:rosgenMS}(\cite[Theorem~3.5]{Ros18})
$\te,\tf,\tk^{\pm 1}$ and $\tell$ generate the $\kk$-algebra $\MS$.
\end{theorem}

\subsection{The Cartan polynomial}\label{subsec:geomrel}

For later use, we now define two polynomials that occur in the presentation. Let $X$ be an indeterminate. In the polynomial ring $\kk[X]$, set
\begin{equation}\label{eq:alpharv}
    \alpha_r=v^{2r-d}\quad(0\le r\le d).
\end{equation}
\begin{align}\label{eq:newrela}
    P_d(X)=\prod_{r=0}^{d}(X-\alpha_r),\qquad
Q_d(X)=\prod_{r=1}^{d}(X-\alpha_r).
\end{align}
When $d=0$, the polynomial $Q_0(X)$ is the empty product and hence equals $1$.

\begin{lemma}\label{lem:geomcalc}
For every $h(X)\in\kk[X]$,
$$
h(\tk)=\sum_{r=0}^{d}h(\alpha_r)\mathbf1_r.
$$
In particular, $P_d(\tk)=0$.
\end{lemma}

\begin{proof}
The orthogonality of the $\mathbf1_r$ in \eqref{eq:orthbfirs} gives
$$
\tk^m=
\left(\sum_{r=0}^{d}\alpha_r\mathbf1_r\right)^m
=\sum_{r=0}^{d}\alpha_r^m\mathbf1_r
$$
for every $m\ge0$. The first identity follows by linearity. Since every $\alpha_r$ is a root of $P_d$, the second identity follows.
\end{proof}

\begin{lemma}\label{lem:geombound}
In $\MS$, we have
$$
\tell\mathbf1_0=\mathbf1_0
$$
and hence
$$
(\tell-1)Q_d(\tk)=0.
$$
\end{lemma}
\begin{proof}
    The definition of $\tell$ gives
$$
\tell\mathbf1_0
=\mathbf1_0^2+
\sum_{r=1}^{d}v^{-2r}
(\mathbf1_r\mathbf1_0+x_r\mathbf1_0).
$$
We have $\mathbf1_0^2=\mathbf1_0$ and $\mathbf1_r\mathbf1_0=0$ for $r\ge1$. 
The basis element $x_r$ is indexed by $D(r,d-r)$, whose column sum vector is $(r,d-r)$. The element $\mathbf1_0$ is indexed by $D(0,d)$, whose row sum vector is $(0,d)$. Since these vectors are different for $r\ge1$, the convolution matching condition gives $x_r\mathbf1_0=0$. This proves the first identity.

By Lemma~\ref{lem:geomcalc}, we have 
$$
Q_d(\tk)=\sum_{s=0}^{d}Q_d(\alpha_s)\mathbf1_s.
$$
The terms with $s\ge1$ vanish. Since $\alpha_s$ are distinct, $c_d:=Q_d(\alpha_0)\ne0$. 
Thus $Q_d(\tk)=c_d\mathbf1_0$, and
$$
(\tell-1)Q_d(\tk)=c_d(\tell-1)\mathbf1_0=0.  
$$
\end{proof}

\section{The finite quotient of mirabolic quantum \texorpdfstring{$\mathfrak{sl}_2$}{sl2}}\label{sec:quot}

\subsection{The universal algebra and the quotient \texorpdfstring{$\textbf{U}(2,d)$}{Ufd}}\label{subsec:univ}

For $m\in\mathbb Z$, set
$$
[m]=\frac{v^m-v^{-m}}{v-v^{-1}}.
$$

\begin{definition}
    The algebra $\Umir$ is the unital associative $\kk$-algebra generated by $e$, $f$, $k$, $k^{-1}$, $\ell$ subject to
    \begin{enumerate}
        \item[(M1)] $kk^{-1}=k^{-1}k=1$,

        \item[(M2)] $kek^{-1}=v^2e$,\qquad $kfk^{-1}=v^{-2}f$,

        \item[(M3)] $ef-fe=(k-k^{-1})/(v-v^{-1})$,

        \item[(M4)] $\ell^2=\ell$,\qquad $k\ell=\ell k$,\qquad $\ell e\ell=\ell e$,\qquad $\ell f\ell=f\ell$,

        \item[(M5)]  ${[2]}e\ell e=v^{-1}e^2\ell+v\ell e^2$,

        \item[(M6)] ${[2]}f\ell f=v^{-1}\ell f^2+vf^2\ell$.
    \end{enumerate}
\end{definition}

This is Rosso's mirabolic quantum $\mathfrak{sl}_2$ \cite[Definition~4.1]{Ros18}. The geometric generators satisfy relations $\textup{(M1)}$-$\textup{(M6)}$, and  \cite[Theorem~3.5 \& Proposition 3.6]{Ros18} gives a natural epimorphism
\begin{equation}\label{eq:epmorMUMS}
    \pi_d:\Umir\twoheadrightarrow\MS.
\end{equation}

The following PBW theorem will be used below.

\begin{theorem}[Rosso]\label{thm:pbw}
The disjoint union of the following six sets is a $\kk$-basis of $\Umir$:
$$
\begin{aligned}
\mathcal B_0&=\{f^ae^bk^t\mid a,b\ge0,\ t\in\mathbb Z\},\\
\mathcal B_1&=\{\ell f^ae^bk^t\mid a,b\ge0,\ t\in\mathbb Z\},\\
\mathcal B_2&=\{f^ae^b\ell k^t\mid a,b\ge0,\ (a,b)\ne(0,0),\ t\in\mathbb Z\},\\
\mathcal B_3&=\{\ell f^ae^b\ell k^t\mid a,b\ge1,\ t\in\mathbb Z\},\\
\mathcal B_4&=\{f^a\ell e^bk^t\mid a,b\ge1,\ t\in\mathbb Z\},\\
\mathcal B_5&=\{e^b\ell f^ak^t\mid a,b\ge1,\ t\in\mathbb Z\}.
\end{aligned}
$$
\end{theorem}

\begin{proof}
This is \cite[Proposition~4.4 and Theorem~4.7]{Ros18}.
\end{proof}

\begin{definition}\label{def:Ufd}
Let $\Ifd$ be the two-sided ideal of $\Umir$ generated by $P_d(k)$ and $(\ell-1)Q_d(k)$. We define
$$
\Ufd=\Umir/\Ifd.
$$
We use the same letters $e,f,k,k^{-1},\ell$ for the images of the generators in $\Ufd$.
\end{definition}

\begin{proposition}\label{prop:factor}
The epimorphism $\pi_d$ in \eqref{eq:epmorMUMS} factors uniquely through an epimorphism
\begin{equation}\label{eq:factorMUMS}
    \overline\pi_d:\Ufd\twoheadrightarrow\MS
\end{equation}
which sends $e,f,k^{\pm 1},\ell$ to $\te,\tf,{\tk}^{\pm 1},\tell$, respectively.
\end{proposition}

\begin{proof}
Lemmas~\ref{lem:geomcalc} and \ref{lem:geombound} show that
$$
P_d(\tk)=0,
\qquad
(\tell-1)Q_d(\tk)=0.
$$
Thus, we have $\Ifd\subseteq\Ker\pi_d$. 
Now, the assertion follows from the universal property of quotient algebras.
\end{proof}

\subsection{Lagrange idempotents}\label{subsec:lagrange}

Since $v$ is an indeterminate, the elements $\alpha_0,\ldots,\alpha_d$ are pairwise distinct. 
For $0\le r\le d$, define
$$
L_r(X)=
\prod_{\substack{0\le s\le d\\s\ne r}}
\frac{X-\alpha_s}{\alpha_r-\alpha_s}
\in\kk[X].
$$
Then $L_r(\alpha_s)=\delta_{r,s}$.

\begin{lemma}\label{lem:poly}
For $0\le r,s\le d$, each of the following polynomials is divisible by $P_d(X)$:
$$
L_r(X)L_s(X)-\delta_{r,s}L_r(X),\quad \sum_{r=0}^{d}L_r(X)-1,\quad\text{and}\quad (X-\alpha_r)L_r(X).
$$
\end{lemma}
\begin{proof}
A polynomial in $\kk[X]$ is divisible by $P_d(X)$ if and only if it vanishes at all the distinct roots $\alpha_0,\ldots,\alpha_d$. 
For $0\le t\le d$, at $X=\alpha_t$, the first polynomial takes the value
$$
\delta_{r,t}\delta_{s,t}-\delta_{r,s}\delta_{r,t}=0,
$$
the second polynomial takes the value $1-1=0$, and the third polynomial takes the value $(\alpha_t-\alpha_r)\delta_{r,t}=0$.
The result follows.
\end{proof}

\begin{definition}\label{def:eps}
For $0\le r\le d$, define
$$
\eps_r=L_r(k)\in\Ufd.
$$
We set $\eps_j=0$ when $j\notin\{0,1,\ldots,d\}$.
\end{definition}

\begin{lemma}\label{lem:eps}
In $\Ufd$, we have
$$
\eps_r\eps_s=\delta_{r,s}\eps_r,
\qquad
\sum_{r=0}^{d}\eps_r=1,
\qquad
(k-\alpha_r)\eps_r=0.
$$
Moreover,
$$
k=\sum_{r=0}^{d}\alpha_r\eps_r,
\qquad
k^{-1}=\sum_{r=0}^{d}\alpha_r^{-1}\eps_r,
\qquad
\ell\eps_r=\eps_r\ell.
$$
\end{lemma}
\begin{proof}
Substitute $X=k$ in Lemma~\ref{lem:poly} and use $P_d(k)=0$. This gives the first three identities. Since $\sum_r\eps_r=1$,
$$
k=\sum_{r=0}^{d}k\eps_r
=\sum_{r=0}^{d}\alpha_r\eps_r.
$$
Multiplying $(k-\alpha_r)\eps_r=0$ by $k^{-1}$ gives $k^{-1}\eps_r=\alpha_r^{-1}\eps_r$, and the formula for $k^{-1}$ follows. Finally, $\ell$ commutes with $k$ and hence with every polynomial in $k$.
\end{proof}

\begin{lemma}\label{lem:boundary}
Let $c_d=Q_d(\alpha_0)$. Then $c_d\ne0$ and
$$
Q_d(k)=c_d\eps_0.
$$
Consequently, the relation $(\ell-1)Q_d(k)=0$ is equivalent to
$$
(\ell-1)\eps_0=0.
$$
\end{lemma}

\begin{proof}
Similar to Lemma \ref{lem:geomcalc}, Lemma~\ref{lem:eps} gives that $h(k)=\sum_{s=0}^{d}h(\alpha_s)\eps_s$ for every $h(X)\in\kk[X]$. 
Taking $h(X)=Q_d(X)$, the values $Q_d(\alpha_s)$ vanish for $s\ge1$, while $Q_d(\alpha_0)\ne0$. Hence $Q_d(k)=c_d\eps_0$. Since $c_d$ is invertible in $\kk$, the two relations are equivalent.
\end{proof}

\subsection{Weight shifts and finite dimension}\label{subsec:finite}

\begin{lemma}\label{lem:components}
For $0\le r,s\le d$,
$$
(\alpha_s-v^2\alpha_r)\eps_se\eps_r=0,
\qquad
(\alpha_s-v^{-2}\alpha_r)\eps_sf\eps_r=0.
$$
\end{lemma}

\begin{proof}
Since $\eps_s$ is a polynomial in $k$, we have $k\eps_s=\eps_sk=\alpha_s\eps_s$. Using $ke=v^2ek$, we obtain
$$
k(\eps_se\eps_r)=\eps_ske\eps_r
=v^2\eps_se k\eps_r=v^2\alpha_r\eps_se\eps_r.
$$
The left-hand side is also $\alpha_s\eps_se\eps_r$. This proves the first identity. 
The second identity follows in the same way from $kf=v^{-2}fk$.
\end{proof}

\begin{lemma}\label{lem:shift}
For $0\le r\le d$, we have
$$
e\eps_r=\eps_{r+1}e,
\qquad
f\eps_r=\eps_{r-1}f.
$$
In particular,
$$
e\eps_d=0,
\qquad
f\eps_0=0.
$$
\end{lemma}
\begin{proof}
Fix $0\le r<d$. Since $1=\sum_s\eps_s$,
$$
e\eps_r=\sum_{s=0}^{d}\eps_se\eps_r.
$$
Note that $v^2\alpha_r=\alpha_{r+1}$. Lemma~\ref{lem:components} and the distinctness of $\alpha_s$ show that every summand $\eps_se\eps_r$ is zero except the one with $s=r+1$. Thus
$$
e\eps_r=\eps_{r+1}e\eps_r.
$$
On the other hand,
$$
\eps_{r+1}e=\sum_{t=0}^{d}\eps_{r+1}e\eps_t.
$$
The same calculation shows that only the term with $t=r$ can be nonzero. Hence
$$
\eps_{r+1}e=\eps_{r+1}e\eps_r=e\eps_r.
$$
If $r=d$, then $v^2\alpha_d=v^{d+2}$ is not among the roots $\alpha_s$. Lemma~\ref{lem:components} gives $\eps_se\eps_d=0$ for every $s$, and hence $e\eps_d=0=\eps_{d+1}e$.

The proof for $f$ is the same. For $1\le r\le d$, one uses $v^{-2}\alpha_r=\alpha_{r-1}$. For $r=0$, the value $v^{-2}\alpha_0=v^{-d-2}$ is not among the roots. Thus $f\eps_0=0=\eps_{-1}f$.
\end{proof}

\begin{corollary}\label{cor:nil}
In $\Ufd$, we have
$$
e^{d+1}=0,
\qquad
f^{d+1}=0.
$$
\end{corollary}
\begin{proof}
Repeated use of Lemma~\ref{lem:shift} gives
$$
e^a\eps_r=\eps_{r+a}e^a,
\qquad
f^a\eps_r=\eps_{r-a}f^a
$$
for $a\ge0$ and $0\le r\le d$. Take $a=d+1$. Then $\eps_{r+d+1}=0$ and $\eps_{r-d-1}=0$. Hence
$$
e^{d+1}\eps_r=0,
\qquad
f^{d+1}\eps_r=0
$$
for every $r$. Summing over $r$ proves the result.
\end{proof}

\begin{proposition}\label{prop:finite}
The algebra $\Ufd$ is finite-dimensional over $\kk$.
\end{proposition}
\begin{proof}
Take the images in $\Ufd$ of the PBW basis in Theorem~\ref{thm:pbw}. Corollary~\ref{cor:nil} allows the exponents of $e$ and $f$ to be restricted to $0,1,\ldots,d$. 
Since $P_d(X)$ is a polynomial of degree $d+1$ with nonzero constant term $(-1)^{d+1}$, 
the relation $P_d(k)=0$ reduces every $k^t$, with $t\in\mathbb Z$, to a linear combination of $1,k,\ldots,k^d$. 
The six PBW families contain only finitely many possible positions of $\ell$. 
Their images therefore give a finite spanning set of $\Ufd$.
\end{proof}

\section{Simple modules for \texorpdfstring{$\mathbf U(2,d)$}{U(2,d)}}\label{sec:simple}
\subsection{Rosso's simple modules}\label{subsec:rossomod}

Let $U_v(\mathfrak{sl}_2)$ be the subalgebra of $\Umir$ generated by $e$, $f$, $k$, $k^{-1}$. For $\delta\in\{0,1\}$, let
$$
\pi_\delta:\Umir\longrightarrow U_v(\mathfrak{sl}_2)
$$
be the epimorphism that fixes $e,f,k,k^{-1}$ and sends $\ell$ to $\delta$. Let $L^{\pm}(m)$ be the simple $U_v(\mathfrak{sl}_2)$-module of highest $k$-weight $\pm v^m$, where $m\in\mathbb N$ (cf. \cite{Jan96}).

We define
$$
L^{\pm}(m,0)=\pi_0^*L^{\pm}(m),
\qquad
L^{\pm}(m,1)=\pi_1^*L^{\pm}(m).
$$
Thus, $L^{\pm}(m,0)$ and $L^{\pm}(m,1)$ are the pullbacks of the standard $(m+1)$-dimensional simple $U_v(\mathfrak{sl}_2)$-module along $\pi_0$ and $\pi_1$, respectively. 
The element $\ell$ acts as $0$ on the first module and as the identity on the second.
Their $k$-eigenvalues are
$$
\{\pm v^{m-2i}\mid0\le i\le m\}.
$$

For $m\ge1$, \cite[Proposition~5.4]{Ros18} also defines the simple modules $L^{\pm}(m,01)$ with basis
$$
\{u^{\pm}_{i,0}\mid0\le i\le m-1\}
\ \sqcup\
\{u^{\pm}_{j,1}\mid1\le j\le m\}.
$$
For every basis vector $u^{\pm}_{i,\epsilon}$ occurring above, where $\epsilon\in\{0,1\}$, the actions of $k$ and $\ell$ are
$$
ku^{\pm}_{i,\epsilon}=\pm v^{m-2i}u^{\pm}_{i,\epsilon},
\qquad
\ell u^{\pm}_{i,\epsilon}=\epsilon u^{\pm}_{i,\epsilon}.
$$
The actions of $e$ and $f$ are explicitly given in \cite[Proposition~5.4]{Ros18}. In particular, the highest $k$-eigenspace of $L^+(m,01)$ is $\kk u^+_{0,0}$, while the lowest $k$-eigenspace is $\kk u^+_{m,1}$. Moreover, $\dim_{\kk}L^{\pm}(m,01)=2m$.

\begin{theorem}[Rosso]\label{thm:rossosimple}
The following modules form a complete set of pairwise non-isomorphic finite-dimensional simple $\Umir$-modules:
$$
\mathrm{Irr}(\Umir)=\{L^{\pm}(m,0),\,\,L^{\pm}(m,1)\mid m\in\mathbb N\} \sqcup \{L^{\pm}(m,01)\mid m\in\mathbb N_{>0}\}.
$$
Moreover, every finite-dimensional $\Umir$-module is a direct sum of simple modules.
\end{theorem}
\begin{proof}
The classification is \cite[Theorem~5.8]{Ros18}, and the semisimplicity is \cite[Theorem~5.13]{Ros18}.  A version of general rank is given in \cite{GR26}.
\end{proof}

\subsection{Modules satisfying the Cartan relation}\label{subsec:kmodules}

Let $M$ be a $\Umir$-module. We say that $M$ satisfies $P_d(k)=0$ if
$$
P_d(k)z=0
$$
for every $z\in M$. This means that the operator induced by $P_d(k)$ is zero on $M$.

\begin{proposition}\label{prop:kdecomp}
Suppose that a $\Umir$-module $M$ satisfies the Cartan relation $P_d(k)=0$. For $0\le r\le d$, set
$$
M_r=\Ker(k-\alpha_r).
$$
Then
$$
M=\bigoplus_{r=0}^{d}M_r,
\qquad
\eps_rM=M_r.
$$
\end{proposition}
\begin{proof}
The polynomial identities in Lemma~\ref{lem:poly} act on $M$. Hence
$$
\eps_r\eps_s=\delta_{r,s}\eps_r,
\qquad
\sum_{r=0}^{d}\eps_r=1,
\qquad
(k-\alpha_r)\eps_r=0
$$
as endomorphisms of $M$. It follows that $M=\sum_r\eps_rM$ and $\eps_rM\subseteq M_r$.

Conversely, take $z\in M_r$. Since $kz=\alpha_rz$, polynomial evaluation gives
$$
\eps_sz=L_s(k)z=L_s(\alpha_r)z=\delta_{s,r}z.
$$
Thus $z=\eps_rz$, so $M_r\subseteq\eps_rM$. The equality follows. If $z_0+\cdots+z_d=0$ with $z_r\in M_r$, applying $\eps_s$ gives $z_s=0$. Hence, the sum is direct.
\end{proof}

Thus, $\eps_r$ is the Lagrange idempotent projecting $M$ onto the $\alpha_r$-eigenspace of $k$. In particular, $\eps_0M=\Ker(k-v^{-d})$.

Define
\begin{equation}\label{eq:calId}
    \Iset_d=
\{m\in\mathbb N\mid0\le m\le d,\ m\equiv d\pmod2\},
\end{equation}
and set $\Iset_t=\varnothing$ for $t<0$.

\begin{lemma}\label{lem:cartanmodules}
Let $\bullet\in \{0,1,01\}$. 
A positive simple module $L^+(m,\bullet)$ satisfies $P_d(k)=0$ if and only if $m\in\Iset_d$. No negative simple module $L^-(m,\bullet)$ satisfies $P_d(k)=0$.
\end{lemma}

\begin{proof}
On a positive simple module, the set of $k$-eigenvalues is
$$
\{v^m,v^{m-2},\ldots,v^{-m}\}.
$$
If these eigenvalues are roots of $P_d$, then $v^m=v^{2r-d}$ for some $0\le r\le d$. Since $v$ is an indeterminate, $m=2r-d$. Thus $m\le d$ and $m\equiv d\pmod2$.

Conversely, suppose that $m=d-2t$ for some $t\ge0$. For $0\le i\le m$, set $r=d-t-i$. Then $0\le r\le d$ and
$$
2r-d=m-2i.
$$
Every $k$-eigenvalue is therefore a root of $P_d$.

On a negative simple module, every $k$-eigenvalue has the form $-v^a$, while every root of $P_d$ has the form $v^b$. The equality $-v^a=v^b$ cannot hold in $\mathbb C(v)$. Thus, no negative simple module satisfies the relation.
\end{proof}

\subsection{The second relation and the complete list}\label{subsec:boundarysimple}

By Lemma~\ref{lem:boundary}, the second relation is equivalent to
$$
\ell\eps_0=\eps_0.
$$
On a module satisfying $P_d(k)=0$, 
Proposition~\ref{prop:kdecomp} identifies $\eps_0M$ with $\Ker(k-v^{-d})$. Hence, the second relation requires $\ell$ to act as the identity on this eigenspace.

\begin{lemma}\label{lem:lowspace}
Let $m\in\Iset_d$.
\begin{enumerate}
\item[(1)] If $m<d$, then $\Ker(k-v^{-d})=0$ on each of $L^+(m,0)$, $L^+(m,1)$, and $L^+(m,01)$.
\item[(2)] If $m=d$, then $\eps_0$ projects $L^+(d,0)$ and $L^+(d,1)$ onto their one-dimensional $v^{-d}$-eigenspaces. It projects $L^+(d,01)$ onto $\kk u^+_{d,1}$.
\end{enumerate}
\end{lemma}
\begin{proof}
The eigenvalues on each positive simple module have the form $v^{m-2i}$ with $0\le i\le m$. If $v^{m-2i}=v^{-d}$, then $i=(m+d)/2$. When $m<d$, this number is greater than $m$, so no such eigenvalue occurs. This proves part~(1).

If $m=d$, the eigenvalue $v^{-d}$ is the lowest eigenvalue of the standard $(d+1)$-dimensional module and its eigenspace is one-dimensional. 
The mixed module $L^+(d,01)$ has basis $\{u^+_{i,0}\mid0\le i\le d-1\}
\sqcup
\{u^+_{j,1}\mid1\le j\le d\}$. 
Thus, the index $i=d$ occurs only in the vector $u^+_{d,1}$. This proves part~(2).
\end{proof}

\begin{theorem}\label{thm:simples}
Up to isomorphism, the simple $\Ufd$-modules are precisely
\begin{align*}
\mathrm{Irr}(\Ufd)&=\{L^+(m,1)\mid m\in\Iset_d\}\\
&\quad \sqcup \{L^+(m,01)\mid m\in\Iset_d,\ m\ge1\}\\
&\quad \sqcup \{L^+(m,0)\mid m\in\Iset_{d-2}\}.
\end{align*}
\end{theorem}
\begin{proof}
Proposition~\ref{prop:finite} shows that every simple $\Ufd$-module is finite-dimensional. Pulling it back along $\Umir\twoheadrightarrow\Ufd$ gives a simple $\Umir$-module. The module therefore belongs to Rosso's list $\mathrm{Irr}(\Umir)$ in Theorem~\ref{thm:rossosimple}.

Lemma~\ref{lem:cartanmodules} excludes all negative modules and restricts the positive parameter to $m\in\Iset_d$. It remains to test $\ell\eps_0=\eps_0$. If $m<d$, Lemma~\ref{lem:lowspace} gives $\eps_0=0$, so the relation holds automatically.

Suppose that $m=d$. On $L^+(d,0)$, the element $\ell$ acts as zero. Since $\eps_0$ has nonzero image, this module does not satisfy $\ell\eps_0=\eps_0$. On $L^+(d,1)$, the element $\ell$ acts as the identity, so the relation holds. Finally, on $L^+(d,01)$ the image of $\eps_0$ is $\kk u^+_{d,1}$, and $\ell u^+_{d,1}=u^+_{d,1}$. 
Hence, the second relation holds on this module. 
Thus, the $0$-type modules are indexed by $\Iset_d\setminus\{d\}=\Iset_{d-2}$, while the other two types retain all allowed parameters.

Conversely, every module in the displayed list satisfies $P_d(k)=0$ by Lemma~\ref{lem:cartanmodules}. The preceding argument also proves $\ell\eps_0=\eps_0$. Hence the ideal $\Ifd$ acts as zero, so the $\Umir$-action factors uniquely through $\Ufd$. Simplicity is unchanged under this factorisation because $\Umir\twoheadrightarrow\Ufd$ is surjective.
\end{proof}

\section{Split semisimplicity and the presentation theorem}\label{sec:main}
\subsection{Split semisimplicity}\label{subsec:split}

\begin{proposition}\label{prop:semisimple}
The algebra $\Ufd$ is a finite-dimensional semisimple $\kk$-algebra.
\end{proposition}

\begin{proof}
Let $M$ be a finite-dimensional $\Ufd$-module. Pull it back to a $\Umir$-module. By Theorem~\ref{thm:rossosimple},
$$
M=\bigoplus_j S_j
$$
as a direct sum of simple $\Umir$-modules. The ideal $\Ifd$ annihilates $M$, and hence annihilates every $S_j$. Each $S_j$ is therefore a $\Ufd$-module, and the same direct sum is a decomposition by $\Ufd$-submodules. Thus, every finite-dimensional $\Ufd$-module is semisimple. In particular, the left regular module is semisimple, so $\Ufd$ is a semisimple algebra.
\end{proof}

\begin{lemma}\label{lem:end}
For every simple $\Ufd$-module $S$ in Theorem~\ref{thm:simples}, we have
$$
\End_{\Ufd}(S) \cong \kk.
$$
\end{lemma}
\begin{proof}
For $\lambda\in\kk$ and $\delta\in\{0,1\}$, write joint eigenspaces (cf. \cite[Definition 5.1]{Ros18})
$$
S_{\lambda,\delta}
=\{x\in S\mid kx=\lambda x,\ \ell x=\delta x\}.
$$
The three types of simple modules in Theorem~\ref{thm:simples} have the following one-dimensional joint eigenspaces at the highest $k$-eigenvalue:
$$
L^+(m,0)_{v^m,0},
\qquad
L^+(m,1)_{v^m,1},
\qquad
L^+(m,01)_{v^m,0}=\kk u^+_{0,0}.
$$
Let $\varphi\in\End_{\Ufd}(S)$. Since $\varphi$ is a module homomorphism, it commutes with the actions of $k$ and $\ell$, and preserves each joint eigenspace. Hence $\varphi(u)=cu$ for a nonzero vector $u$ in the displayed one-dimensional space and some $c\in\kk$. The simple module $S$ is generated by $u$. Therefore
$$
\varphi(au)=a\varphi(u)=cau
$$
for every $a\in\Ufd$, and $\varphi=c\operatorname{id}_S$.
\end{proof}

\begin{corollary}\label{cor:split}
The algebra $\Ufd$ is split semisimple over $\kk$.
\end{corollary}

\begin{proof}
A finite-dimensional semisimple algebra is split if the endomorphism division algebra of every simple module is the ground field. The result follows from Proposition~\ref{prop:semisimple} and Lemma~\ref{lem:end}.
\end{proof}

\subsection{The Wedderburn decomposition and the dimension}\label{subsec:wed}

\begin{theorem}\label{thm:wed}
There is a $\kk$-algebra isomorphism
$$
\Ufd\cong
\bigoplus_{m\in\Iset_d}\Mat_{m+1}(\kk)
\ \oplus\
\bigoplus_{m\in\Iset_d,\, m\ge1}\Mat_{2m}(\kk)
\ \oplus\
\bigoplus_{m\in\Iset_{d-2}}\Mat_{m+1}(\kk).
$$
The three sums correspond, in order, to the simple modules of types $(m,1)$, $(m,01)$, and $(m,0)$. Moreover,
$$
\dim_{\kk}\Ufd=(d+1)^3.
$$
\end{theorem}
\begin{proof}
Corollary~\ref{cor:split} and the Artin-Wedderburn Theorem identify one full matrix block for each simple module. The dimensions of the three types are
$$
\dim L^+(m,1)=m+1,
\qquad
\dim L^+(m,01)=2m,
\qquad
\dim L^+(m,0)=m+1.
$$
The stated decomposition follows from Theorem~\ref{thm:simples}.

Let $D_d=\dim_{\kk}\Ufd$. Then
$$
D_d=
\sum_{m\in\Iset_d}(m+1)^2
+
\sum_{m\in\Iset_d,\, m\ge1}(2m)^2
+
\sum_{m\in\Iset_{d-2}}(m+1)^2.
$$
We have $D_0=1$ and $D_1=8$. For $d\ge2$,
$$
\Iset_d=\Iset_{d-2}\sqcup\{d\}.
$$
Using $\Iset_d=\Iset_{d-2}\sqcup\{d\}$ in the three sums for $D_d$ and $D_{d-2}$ gives
$$
\begin{aligned}
D_d-D_{d-2}
&=(d+1)^2+(2d)^2+(d-1)^2=6d^2+2\\
&=(d+1)^3-(d-1)^3.
\end{aligned}
$$
Induction along the even and odd values of $d$ gives $D_d=(d+1)^3$.
\end{proof}

\begin{example}[Low-degree cases]\label{ex:small-d}
Theorem~\ref{thm:wed} gives
$$
\mathbf U(2,0)\cong\kk,\qquad \mathbf U(2,1)\cong
\Mat_2(\kk)\oplus\Mat_2(\kk),
$$
$$
\mathbf U(2,2)\cong
\Mat_3(\kk)\oplus\Mat_1(\kk)\oplus
\Mat_4(\kk)\oplus\Mat_1(\kk).
$$
The corresponding dimensions are $1$, $8$, and $27$. For $d=0$, one has $P_0(k)=k-1$ and $Q_0(k)=1$, so $k=\ell=1$. 
Relations \textup{(M2)} then give $e=v^2e$ and $f=v^{-2}f$. Hence $e=f=0$, in agreement with $\mathbf U(2,0)\cong\kk$.
\end{example}

\subsection{The main isomorphism}\label{subsec:isom}

\begin{theorem}\label{thm:mainiso}
The epimorphism
$$
\overline\pi_d:\Ufd\twoheadrightarrow\MS
$$
from Proposition~\ref{prop:factor} is an isomorphism.
\end{theorem}

\begin{proof}
Theorem~\ref{thm:wed} and Lemma~\ref{lem:geomdim} give
$$
\dim_{\kk}\Ufd=(d+1)^3=\dim_{\kk}\MS.
$$
A surjective linear map between finite-dimensional vector spaces of the same dimension is bijective. Hence $\overline\pi_d$ is a $\kk$-algebra isomorphism.
\end{proof}

\begin{corollary}\label{cor:kernel}
The natural epimorphism $\pi_d:\Umir\twoheadrightarrow\MS$ satisfies
$$
\Ker\pi_d=
\left\langle P_d(k),(\ell-1)Q_d(k)\right\rangle.
$$
\end{corollary}

\begin{proof}
The ideal on the right is contained in $\Ker\pi_d$ by Proposition~\ref{prop:factor}. The induced map from the quotient is injective by Theorem~\ref{thm:mainiso}. Therefore, no further element lies in the kernel.
\end{proof}

\section{An idempotent presentation}\label{sec:idemp}
The Cartan generator can be removed from the presentation by taking the weight idempotents as generators.

\begin{definition}\label{def:idemp}
Let $\mathbf U^{\mathrm{id}}(2,d)$ be the unital $\kk$-algebra generated by
$$
\be,\quad\bfm,\quad\bell,\quad
\beps_0,\quad\beps_1,\quad\ldots,\quad\beps_d.
$$
We set $\beps_j=0$ for $j\notin\{0,1,\ldots,d\}$. The defining relations are
$$
\beps_r\beps_s=\delta_{r,s}\beps_r,
\qquad
\sum_{r=0}^{d}\beps_r=1
\qquad(0\le r,s\le d),
$$
$$
\be\beps_r=\beps_{r+1}\be,
\qquad
\bfm\beps_r=\beps_{r-1}\bfm
\qquad(0\le r\le d),
$$
$$
\be\bfm-\bfm\be
=\sum_{r=0}^{d}[2r-d]\beps_r,
$$
$$
\bell^2=\bell,
\qquad
\bell\beps_r=\beps_r\bell
\qquad(0\le r\le d),
\qquad
\bell\beps_0=\beps_0,
$$
$$
\bell\be\bell=\bell\be,
\qquad
\bell\bfm\bell=\bfm\bell,
$$
and
$$
[2]\be\bell\be
=v^{-1}\be^2\bell+v\bell\be^2,\qquad [2]\bfm\bell\bfm
=v^{-1}\bell\bfm^2+v\bfm^2\bell.
$$
\end{definition}

\begin{theorem}\label{thm:idemp}
There is a unique $\kk$-algebra isomorphism
$$
\mathbf U^{\mathrm{id}}(2,d)
\xrightarrow{\ \sim\ }
\Ufd
$$
which sends
$$
\be\mapsto e,
\qquad
\bfm\mapsto f,
\qquad
\bell\mapsto\ell,
\qquad
\beps_r\mapsto\eps_r
\quad(0\le r\le d).
$$
\end{theorem}
\begin{proof}
    Work first in $\mathbf U^{\mathrm{id}}(2,d)$. Define
$$
K=\sum_{r=0}^{d}\alpha_r\beps_r,
\qquad
K^{-1}=\sum_{r=0}^{d}\alpha_r^{-1}\beps_r.
$$
The orthogonal idempotent relations give $KK^{-1}=K^{-1}K=1$.

For $0\le r<d$,
$$
K\be\beps_r
=\alpha_{r+1}\be\beps_r
=v^2\alpha_r\be\beps_r
=v^2\be K\beps_r.
$$
For $r=d$, both sides are zero because $\be\beps_d=0$. Summing over $r$ gives $K\be K^{-1}=v^2\be$. The same argument gives $K\bfm K^{-1}=v^{-2}\bfm$.

Furthermore, $\frac{\alpha_r-\alpha_r^{-1}}{v-v^{-1}}=[2r-d]$. 
Hence the commutator relation in Definition~\ref{def:idemp} is equivalent to
$$
\be\bfm-\bfm\be
=\frac{K-K^{-1}}{v-v^{-1}}.
$$
The remaining relations show that the assignments
$$
e\mapsto\be,
\quad
f\mapsto\bfm,
\quad
k\mapsto K,
\quad
k^{-1}\mapsto K^{-1},
\quad
\ell\mapsto\bell
$$
define a homomorphism from $\Umir$ to $\mathbf U^{\mathrm{id}}(2,d)$.

For every $h(X)\in\kk[X]$, orthogonality gives $h(K)=\sum_{r=0}^{d}h(\alpha_r)\beps_r$. 
Thus $P_d(K)=0$ and
$$
Q_d(K)=Q_d(\alpha_0)\beps_0.
$$
Since $\bell\beps_0=\beps_0$, we have $(\bell-1)Q_d(K)=0$. The homomorphism therefore factors through a homomorphism
$$
\Psi:\Ufd\longrightarrow\mathbf U^{\mathrm{id}}(2,d).
$$

Conversely, Lemmas~\ref{lem:eps}, \ref{lem:boundary}, and \ref{lem:shift}, together with the universal relations in $\Umir$, show that the elements
$$
e,\quad f,\quad\ell,\quad\eps_0,\ldots,\eps_d
$$
in $\Ufd$ satisfy all the relations in Definition~\ref{def:idemp}. Hence, there is a homomorphism
$$
\Theta:\mathbf U^{\mathrm{id}}(2,d)\longrightarrow\Ufd.
$$

The composite $\Theta\Psi$ fixes $e,f,\ell$ and sends $k$ to
$$
\sum_{r=0}^{d}\alpha_r\eps_r=k
$$
by Lemma~\ref{lem:eps}. It is therefore the identity on $\Ufd$. The other composite $\Psi\Theta$ fixes $\be,\bfm,\bell$ and, for $0\le r\le d$, satisfies
$$
\Psi\Theta(\beps_r)
=L_r(K)
=\sum_{s=0}^{d}L_r(\alpha_s)\beps_s
=\beps_r.
$$
Thus, $\Psi$ and $\Theta$ are inverse isomorphisms.
\end{proof}

\end{document}